\documentclass[12pt]{amsart}

\usepackage{amsmath}
\usepackage{amsfonts}
\usepackage{amssymb,enumerate}
\usepackage{amsthm}

\newtheorem{lem}{Lemma}[section]
\newtheorem{cor}[lem]{Corollary}
\newtheorem{prop}[lem]{Proposition}
\newtheorem{thm}[lem]{Theorem}

\newtheorem{defin}[lem]{Definition}
\newtheorem{Ex}[lem]{Example}
\newtheorem{Remark}[lem]{Remark}
\newtheorem{Construction}[lem]{Construction}
\newtheorem{Notation}[lem]{Notation}
\newtheorem{Fact}[lem]{Fact}
\newtheorem{Discussion}[lem]{Discussion}
\newtheorem{Notationdefinition}[lem]{Definition/Notation}
\newtheorem{Remarkdefinition}[lem]{Remark/Definition}

\newenvironment{ex}{\begin{Ex}\rm}{\end{Ex}}

\newenvironment{remark}{\begin{Remark}\rm}{\end{Remark}}
\newenvironment{disc}{\begin{Discussion}\rm}{\end{Discussion}}

\newcommand{\PP}{\mathbb{P}}
\newcommand{\pperp}{\mathrel{\mbox{$\perp\hspace{-0.6em}\perp$}}}

\newcommand{\tensor}{\otimes}

\def\switch{{\mathop{\rm s}\nolimits}}

\def\It{{\widetilde{I}}}

\def\I{{I}}

\def\CI{\I^{\langle t \rangle}}
\def\CIo{\I^{\langle 1 \rangle}}
\def\CIn{\I^{\langle n \rangle}}
\def\CInmin{\I^{\langle n-1 \rangle}}
\def\CIt{\It^{\langle t \rangle}}

\def\CItn{\It^{\langle n \rangle}}
\def\CIS{\I^{\langle t \rangle}_S}

\def\CItS{\It^{\langle t \rangle}_S}

\def\CInN{\It^{\langle n \rangle}_N}
\def\CInS{\It^{\langle n \rangle}_S}
\def\CInSP{\It^{\langle n \rangle}_{S'}}

\def\CItSi{\It^{\langle t \rangle}_{S_i}}

\def\CPS{P^{\langle t \rangle}_S}

\def\CPSn{P^{\langle n \rangle}_S}
\def\CPSnP{P^{\langle n \rangle}_{S'}}
\def\CPT{P^{\langle t \rangle}_T}
\def\CPN{P^{\langle t \rangle}_N}
\def\CPnN{P^{\langle n \rangle}_N}

\def\CVar{\mathrm{Var}^{\langle t \rangle}_S}

\def\CVarn{\mathrm{Var}^{\langle n \rangle}_S}
\def\CVarnP{\mathrm{Var}^{\langle n \rangle}_{S'}}
\def\CVarT{\mathrm{Var}^{\langle t \rangle}_T}

\begin{document}


\bibliographystyle{amsplain}

\title{Minimal primes of ideals arising from conditional independence statements}

\author[Swanson]{Irena Swanson}
\address[Swanson]{Department of Mathematics, Reed College, 3203
SE Woodstock Blvd, Portland, OR 97202, USA}
\email{iswanson@reed.edu}
\author[Taylor]{Amelia Taylor}
\address[Taylor]{Department of Mathematics and Computer
Science, Colorado College, 14 E. Cache La Poudre St.,
Colorado Springs, CO 80903, USA}
\email{amelia.taylor@coloradocollege.edu}



\begin{abstract}
We consider ideals arising in the context of conditional independence
models that generalize the class of ideals considered by Fink~\cite{Fink}
in a way distinct from the generalizations of
Herzog-Hibi-Hreinsdottir-Kahle-Rauh~\cite{HHHKR} and
Ay-Rauh~\cite{AR}.
We introduce switchable sets to give a combinatorial description of the minimal prime ideals,
and for some classes we describe the minimal components.
We discuss many possible interpretations of the ideals we study,
including as $2 \times 2$ minors of generic hypermatrices.  
We also introduce a definition of diagonal monomial orders
on generic hypermatrices
and we compute some Gr\"obner bases.

\end{abstract}

\maketitle

\section{Introduction} \label{intro}

We present work on the primary decomposition of ideals
corresponding to conditional independence models.
One of our motivations is the work of Fink~\cite{Fink}
which solved a conjecture
posed by Cartwright and Engstr\"om (page 146 in~\cite{DSS}).
Herzog, Hibi, Hreinsdottir, Kahle and Rauh~\cite{HHHKR} were similarly
motivated to study a related set of ideals.
The ideals in~\cite{HHHKR} were independently investigated
by Ohtani~\cite{O},
without regard to conditional independence models,
and Ohtani obtained similar results.
After we completed the first version of this manuscript,
Ay and Rauh~\cite{AR} extended Herzog et al.~\cite{HHHKR}.
We discuss the relations between our results
and those of~\cite{AR, Fink, HHHKR, O}
after we make a few definitions.

Let $X_1, \ldots, X_n$ be $n$
discrete random variables and $1\leq t\leq n$.
The ideal $\CI$,
the main object of our study
and defined algebraically in Definition~\ref{defFIt},
corresponds to the conditional independence model
$$\{X_i\pperp X_j ~|~ X_T : \forall i\leq t, i < j \text{ and } T =
\{1,\ldots, n\}\setminus \{i,j\}\},$$
which is the model given by the set of pairwise Markov
conditions (see~\cite[page 32]{L}) on the graph with no edges on the
first $t$ vertices and a complete graph on the remaining $n-t$ vertices.
For background on conditional independence
models see~\cite{DSS,GSS, GMS, HHHKR, L, Sturm02}.
An easy consequence of the Clifford-Hammersley Theorem~\cite{L},
or Eisenbud and Sturmfels' work~\cite[Corollary 2.5]{ESbinom}, and
stated explicitly in Hosten and Shapiro's work~\cite[Theorem 2.1]{HS},
is that one primary component of the ideal $\CI$ is $\CI:\mathbf{x}^{\infty}$,
where $\mathbf{x}$ is the product of the indeterminates.
Sturmfels calls $\CI:\mathbf{x}^{\infty}$
the ``most important component''~\cite[page 116]{Sturm02},
and Lauritzen~\cite{L} calls it the graphical model.
We summarize several ways of thinking about this ideal
(via tensors, lattices, and hypermatrices) in Section~\ref{setup}
to set up a more in-depth discussion of its
relationship to $\CI$ and we use this discussion in our arguments.

This paper is about the ideals $\CI$
when $X_1, \ldots, X_n$ have an arbitrary but finite number of states,
and $t$ is any positive integer at most $n$.
Fink~\cite{Fink} studied the case $n = 3$ and $t = 1$.
The binomial edge ideals considered by Ohtani~\cite{O}
and by Herzog et al. in~\cite{HHHKR}
intersect our class of ideals when $t = 1$ and $X_1$ is binary.
Ay and Rauh~\cite{AR} generalized~\cite{HHHKR}
and their ideals intersect our class of ideals when $t = 1$.

Our main result, in Section~\ref{sectminprimescont},
is the description of the prime ideals minimal over $\CI$.
In Section~\ref{sectprimecomponent} we prove that the minimal
components of $\CIn$ are prime ideals,
while the work of Ay and Rauh~\cite{AR} establishes that $\CIo$ is
radical and therefore all the primary components are minimal and prime.
We give examples in Section~\ref{sectexamples} showing that
ideals $\CI$ can have embedded primes for $n\geq 3$ and $t > 1$.
Thus the ideals $\CI$ are not radical in general
and therefore are not the same as the ideals in~Herzog et al.~\cite{HHHKR},
Ohtani~\cite{O}, or Ay and Rauh~\cite{AR}.
Example~\ref{exItItt} shows that
the ideals $\CI$ are not lattice basis ideals.

In Section~\ref{sectadm}
we define a new combinatorial structure, \emph{$t$-switchable} set
and a corresponding equivalence relation.
We expect that these structures
might be helpful in other contexts as well.
By~\cite{ESbinom},
the minimal primes of a binomial ideal consist of a set of variables and a
set of binomials in the remaining variables.
The content of
this work is in giving an effective combinatorial description of the
sets of variables and of the binomials,
and we do so by using $t$-switchable sets and the associated equivalence relations.

In Section~\ref{sectminprimescont}
we use a connection with Segre embeddings to prove that $\CItS$
(see Definition~\ref{defPScont})
are prime ideals,
thereby generalizing the fact that
$\CI:\mathbf{x}^{\infty}$ are prime ideals.
It is precisely such ideals $\CItS$, for restricted~$S$,
that are the binomial portion of the prime ideals minimal over $\CI$
(proof is in Theorem~\ref{thmminprimescont}). 

In Section~\ref{sectGB}
we introduce a notion of diagonal monomial orders for generic hypermatrices.  We use these orders to give Gr\"obner bases for $\CItS$.
This generalizes the well-known work
of Caniglia, Guccione, and Guccione~\cite{CGG} for generic matrices
and the work of 
Ha~\cite[Theorem 1.14]{H}
for $\CIn:\mathbf{x}^{\infty}$
in the reverse lexicographic monomial~order.


\section{Definitions and connections with tensors and Segre embeddings}
\label{setup}

After setting up the notation,
we give an algebraic
definition for the conditional independence ideal $\CI$ from the introduction.
One of the prime ideals minimal over~$\CI$
is the ideal corresponding to the model for total independence of variables
(as opposed to conditional independence),
and in Discussion~\ref{discSegre}
we present many fruitful interpretations of this prime ideal.
In Example~\ref{exItItt} we show that $\CI$ are not lattice
basis ideals. 


Throughout we fix positive integers $n$ and $r_1, \ldots, r_n$,
the index set $N = [r_1] \times \cdots \times [r_n]$,
and the polynomial ring $R$ over a field in variables $x_a$
as $a$ varies over $N$.
Let $M$ be the $r_1 \times \cdots \times r_n$ hypermatrix
whose $a$th entry is $x_a$.
In this context, this paper is about the
structure of certain determinantal ideals of this generic hypermatrix~$M$.

Let $L \subseteq [n]$.
For $a, b \in N$ define the {\bf switch} function $\switch(L,a,b)$
that switches the $L$-entries of $b$ into $a$:
$\switch(L,a,b)$ is an element of $N$ whose $i$th component is
$$
\switch(L,a,b)_i = \begin{cases}
b_i, & \text{if } i \in L; \cr
a_i, & \text{otherwise.} \cr
\end{cases}
$$
If $L = \{j\}$,
we simply write
$\switch(L,a,b) = \switch(j,a,b)$.
For any two indices $a$ and $b$ in $N$
we define
{\bf the distance} between them to be
$d(a,b) = \#\{i: a_i \not = b_i\}$. Note that $d(a,b) =
d(\switch(L,a,b), \switch(L,b,a))$.
%
For any $L \subseteq [n]$ and $i \in [n]$ we define:
\begin{align*}
f_{L,a,b} &= x_a x_b - x_{\switch(L,a,b)} x_{\switch(L,b,a)}, \cr
f_{i,a,b} &= x_a x_b - x_{\switch(i,a,b)} x_{\switch(i,b,a)}. \cr
\end{align*}
We call the $f_{i,a,b}$ the {\bf $\bf 2 \times 2$ minors} of the hypermatrix $M$.
When $d(a,b) = 2$ and $a_i \not = b_i$,
we call $f_{i,a,b}$ a {\bf slice minor} of $M$.
By a {\bf slice submatrix} of $M$ we refer to any submatrix of $M$
consisting of all entries $M_{i_1, \ldots, i_n}$
with all but two of the indices identical.
Thus a slice minor of $M$ is simply a $2 \times 2$ minor of
a slice submatrix of $M$.
This notation provides flexibility over flattenings
for discussing certain subsets of the minors of a hypermatrix
such as the slice minors.
We also think of $f_{i,a,b}$ as a minor of a flattening of the hypermatrix,
using the $i$th component to index the rows.
More generally,
$f_{L,a,b}$ is a minor of a flattening of the hypermatrix,
where the rows are indexed by the components in $L$
(see also Discussion~\ref{discSegre}).  

\begin{defin}\label{defFIt}
For any $t \in [n]$, let
\begin{align*}
\CI & = (f_{i,a,b}: a, b \in N, d(a,b) = 2, i \in [t]), \cr
\CIt & = (f_{i,a,b}: a, b \in N, i \in [t]). \cr
\end{align*}
Note that the generators of $\CI$ (resp.\ $\CIt$)
are those slice (not necessarily slice) minors of~$M$
for which one of the two components that varies is $i \in [t]$.
Alternatively,
the generators
of $\CI$ (resp.\ $\CIt$)
are the slice (not necessarily slice) minors
of the generic
$r_1\times r_2\times \cdots \times r_t \times (r_{t+1}\cdots r_n)$ hypermatrix.
The conditional independence model
given at the beginning of the introduction
corresponds to the ideal $\CI$.
\end{defin} 



\begin{disc}\label{discSegre}
Now we connect the ideal $\CIt$,
generated by the $2\times 2$ minors of $M$
in which one of the entries being switched is at most $t$, to
other ideals in the literature to facilitate later arguments.
The following describe the same ideal:
\begin{enumerate}
\item
The ideal cutting out the rank one tensors in the flattenings of
  $V_1\tensor \cdots \tensor V_n$ of the form
  $V_i\tensor (\tensor_{j\neq i} V_j)$ as $i$ varies over $[t]$. 
(For background on tensors we recommend~\cite{B,DL}.)
\item
The defining ideal for the Segre embedding of
  $\PP(V_1)\times\cdots\times\PP(V_t)\times
\PP(V_{t+1}\tensor\cdots\tensor V_n)
\rightarrow\PP(V_1\tensor\cdots \tensor V_n)$.
(For background here we recommend~\cite{B}.)
\item
The ideal generated by all the $2\times 2$ minors of the generic
  $r_1\times r_2\times \cdots\times r_t\times(r_{t+1} \cdots r_n)$ hypermatrix.
\item
The lattice ideal, where the lattice is the kernel of the matrix of
computing the $1$-marginals for each $i  \in [t]$
so that for each possible $i$th state
we marginalize over the remaining variables.
\item
The ideal corresponding to the total independence model on the
  first $t$ variables given the remaining variables.
\end{enumerate}
\end{disc}

The fact that these all define the same ideal is scattered through the
literature, but most of the key ideas are in~\cite{GMS}. For example,
the connection between (2), (4) and (5) follows from~\cite{GMS}
since the model is given as distributions in the image of a
monomial parameterization given by the marginals matrix,
and that monomial
parameterization is exactly the Segre map.
One argument that (2) is equivalent to (3) is in Ha~\cite{H}.
Finally,
it is well known
that a matrix (hypermatrix) is rank one if and only if its $2\times
2$ minors vanish and that such matrices (hypermatrices)
represent rank one tensors in
the corresponding tensor product of vector spaces.

The many interpretations of $\CIt$ have several useful consequences:
\begin{enumerate}
\item
$\CIt$ is prime say by~\cite{H} (and we generalize this fact in Theorem~\ref{thmIStisprime});
\item
$\CI:\mathbf{x}^\infty = \CIt$
(by the Clifford-Hammersley Theorem~\cite{L}, \cite{ESbinom}, or \cite{HS});
\item
therefore $\CIt$ is a minimal prime component of $\CI$,
and
\item
$\CIt$ is the unique smallest binomial prime ideal that contains $\CI$
and that contains no variables.
\end{enumerate}

Since $\CI\subseteq \CIt$ and $\CIt$ is a lattice ideal for a
saturated lattice, it is natural to ask if $\CI$ is a lattice basis
ideal (as defined in~\cite{HS}) for some basis for the same lattice.  The following example
illustrates that this is not the case,
and that the lattice basis ideal is properly contained in~$\CI$.    

\begin{ex}\label{exItItt}
 For a simple illustration of how $\CI$ and $\CIt$ relate to
  lattice ideals, consider the example of three
  random variables, each with two states, and $t = 3$.   The
  $2\times 2\times 2$ hypermatrix has 6 faces and the determinants of
  these faces give six minimal generators of $\CI$.  There are
  six non-slice minors which we add in to generate~$\CIt$
(only three of these are needed to get a minimal generating
  set).  Finally, the lattice basis ideal is minimally generated by
  four binomials, which correspond to four of the six faces (depending
  on which basis one chooses) and is strictly contained in $\CI$.
\end{ex}

Therefore the primary decomposition of $\CI$ does not follow from~\cite{HS}.
By~\cite{ESbinom},
describing the minimal prime ideals of $\CI$
consists of establishing sets of variables and binomials.
We use \emph{$t$-switchable sets} $S$
and ideals $\CItS$ for this purpose (defined in
Section~\ref{sectadm}).  In Section~\ref{sectGB} we place these ideals
in the wider theory and prove that they are prime ideals.

\section{Switchable sets and connectedness}\label{sectadm}

In this section we set up the combinatorial structures used in the
main results. We indicate in Remark~\ref{rmk4Fink} how these
structures relate to those used by
Fink~\cite{Fink}.

\begin{defin}\label{defadmcont}
Let $t \in [n]$.
A subset $S$ of $N$ is {\bf switchable in the first $\bf{t}$ components}
({\bf $\bf t$-switchable} for short)
if
for all $a,b \in S$ with $d(a,b) = 2$,
if $i \in [t]$,
then $\switch(i,a,b) \in S$.
\end{defin}

Certainly the empty set and the full set $N$ are $t$-switchable sets.
Note that the notion of $(n-1)$-switchable is identical to the notion of
$n$-switchable,
and that $t$-switchable is equivalent to the following condition:
For any $a,b \in N$
and any distinct $i \in [t]$ and $j \in [n]$,
$a, \switch(\{i,j\},a,b) \in S$
if and only if $\switch(i,a,b), \switch(j,a,b) \in~S$.

\begin{lem}\label{bootstrappingLemma}
Let $S$ be a $t$-switchable subset of $N$,
let $a \in S$,
and let $b \in N$.
Let $L \subseteq [t]$ be such that for all $l \in L$,
$\switch(l,a,b) \in S$.
Then $\switch(L,a,b)\in S$.
\end{lem}

\begin{proof}
We prove this by induction on $|L|$.
If $|L| \le 1$,
this is the assumption.
If $|L| \ge 2$,
let $i, j$ be distinct elements in $L$.
By induction,
$\switch(L',a,b) \in S$ for all $L' \subseteq L$
with $|L'| < |L|$.
In particular,
if $i, j \in L$ and $L_0 = L \setminus \{i,j\}$, then
$c = \switch(L_0,a,b)$,
$d = \switch(L_0 \cup \{i\},a,b)$,
$e = \switch(L_0 \cup \{j\},a,b) \in S$.
Also,
$\switch(i,c,b) = d$ and $\switch(j,c,b) = e$ are in $S$,
and therefore, since $S$ is $t$-switchable,
$\switch(L,a,b) = \switch(\{i,j\},c,b) \in S$.
\end{proof}

\begin{defin}\label{defconnected}
Let $S$ be a subset of $N$.
We say that $a, b \in S$ are {\bf connected in $\bf S$}
if there exist $a_0 = a, a_1, a_2, \ldots, a_{k-1}, a_k = b \in S$
such that for all $j = 1, \ldots, k$,
$a_{j-1}$ and $a_j$ differ only in one component.
We refer to $a_0, \ldots, a_k$ loosely as a {\bf path} from $a$ to $b$,
and we refer to $a_1, \ldots, a_{k-1}$
as an {\bf intermediate subpath} from $a$ to $b$.
Clearly any elements on the path from $a$ to $b$
are mutually connected.
Also,
connectedness is naturally an equivalence relation on $S$.
\end{defin}


The following is immediate from the definition:

\begin{lem}\label{lemdist2conn}
Let $S$ be a $t$-switchable subset of $N$.
Let $a,b \in S$.
If $d(a,b) \le 1$,
then $a$ and $b$ are connected.
If $d(a,b) = 2$
and $a_i \not = b_i$ for some $i \in [t]$,
then $a$, $b$, $\switch(i,a,b)$, $\switch(i,b,a)$
are pairwise connected in $S$,
and both $\switch(i,a,b)$ and $\switch(i,b,a)$
form an intermediate subpath from $a$ to $b$.
\qed
\end{lem}

\begin{lem}\label{lemconnected}
Suppose that $a$ and $b$ are connected in a $t$-switchable set $S$.
Let $a_0 = a, a_1, a_2$,
$\ldots, a_{k-1}, a_k = b$
be a path from $a$ to $b$.
Let $L \subseteq [t]$ and $i,j \in \{0, \ldots, k\}$.
Then $\switch(L,a_i,a_j) \in S$ and
is connected to $a$ in $S$.
\end{lem}

\begin{proof}
First suppose that $L = \{l\}$.
If $|j - i| \le 2$,
then $\switch(l,a_i,a_j)$ is either $a_i \in S$ or it is in~$S$
by Lemma~\ref{lemdist2conn}.
So we may assume that $|j - i| \ge 3$.
Without loss of generality, assume $i<j$.  By induction on $|j-i|$
we have that $\switch(l,a_{i+1},a_j) \in S$.
Then $a_i, a_{i+1}, \switch(l,a_{i+1},a_j)$ is a path in~$S$, and
therefore by induction,
$\switch(l,a_i, a_j) = \switch(l,a_i, \switch(l,a_{i+1},a_j)) \in S$.
Hence by Lemma~\ref{bootstrappingLemma},
$\switch(L,a_i,a_j) \in S$ for all $L \subseteq [t]$.
Furthermore,
if $L = \{l_1, \ldots, l_k\}$,
then
$a_i, \switch(\{l_1\},a_i,a_j),
\switch(\{l_1,l_2\},a_i,a_j),
\ldots,
\switch(\{l_1,l_2,\ldots, l_k\},a_i,a_j)$
is a path in $S$,
so that $\switch(L,a_i,a_j)$ is connected to $a_i$ and hence to $a$.
\end{proof}

In the definition of connectedness,
the set of indices where the consecutive $a_i$ differ may not all be distinct,
and if $L \not \subseteq [t]$,
then $\switch(L,a_i,a_j)$ need not be connected to $a$,
as we show by the next example.

\begin{ex}\label{exnotblock}
Set
\begin{align*}
S &=
\{1,2\} \times \{1\} \times \{1\} \times \{1,2\} \cr
&\hskip2em
\cup
\{1,2\} \times \{1\} \times \{1,2\} \times \{2\} \cr
&\hskip2em
\cup
\{1,2\} \times \{1,2\} \times \{2\} \times \{2,3\} \cr
&\hskip2em
\cup
\{1,2\} \times \{2\} \times \{2,3\} \times \{2,3\}. \cr
\end{align*}
Note that $S$ is $1$-switchable and consisting of a single equivalence class.
The elements $(1,1,1,1)$ and $(2,2,3,3)$ are connected in $S$,
but there is no path between them of length~$4$.
Also,
$\switch(2,(1,1,1,1),(2,2,3,3)) \not \in S$.
(See the comment after Theorem~\ref{thmblockdec}
for another point of view.)
\end{ex}



\section{Prime ideals minimal over $\CI$}\label{sectminprimescont}

This section has two main goals.
One is to prove (in Theorem~\ref{thmCItSisprime})
that $\CItS$ (see Definition~\ref{defPScont})
are prime ideals.
The second is to prove that
for any $n$ and any $t\in[n]$,
the prime ideals minimal over~$\CI$
are of the form~$\CPS$ as $S$ varies over maximal $t$-switchable subsets of $N$
(see Definition~\ref{defmaxtadm}).
At the end of the section we look at $\CIo$ more closely,
especially when $n = 3$.
We examine the minimal components of $\CI$ when $t = n$
in Section~\ref{sectprimecomponent}.



The following definition gives the notation for the variable and pure binomial
parts of the minimal prime ideals for $\CI$.  

\begin{defin}\label{defPScont}
Let $S$ be $t$-switchable.
Define
$$
\CItS = (f_{i,a,b} \mid i\in [t], a,b\ \text{connected in } S),
\ \CVar = (x_a : a \not \in S),
\text{ and }
\CPS = \CVar + \CItS.
$$
\end{defin}

\begin{prop}\label{propQscontainsCI}
If $S$ is $t$-switchable,
then $\CPS$ contains $\CI$.
\end{prop}

\begin{proof}
We need to prove that $f_{i,a,b} \in \CPS$
for any $a, b \in N$ differing exactly in components~$i \in [t]$
and $j \in [n] \setminus \{i\}$.
First suppose that $a \not \in S$.
Note that $\switch(i,a,b)$ and $\switch(i,b,a)$
differ exactly in components~$i$ and $j$,
so that $S$ being $t$-switchable implies
either $\switch(i,a,b)$ or $\switch(i,b,a)$ is not in $S$.
Thus either
$x_{\switch(i,a,b)}$ or $x_{\switch(i,b,a)}$ is in $\CPS$,
so that $f_{i,a,b} \in \CPS$.
Thus we may assume that $a \in S$,
and similarly that $b \in S$.
But then by Lemma~\ref{lemdist2conn},
$a$ and $b$ are connected in $S$,
so that $f_{i,a,b} \in \CPS$.
\end{proof}

\begin{thm}\label{thmIStisprime}
Let $1 \le t \le n$.  If $S$ is an equivalence class with respect to
connectedness within some $t$-switchable set, then $\CItS$ is a
prime ideal.
\end{thm}

\begin{proof}
Since $S$ is an equivalence class with respect to connectedness
within some $t$-switchable set, for any $a,b\in S$ and any $i\leq t$,
$\switch(i,a,b)$ and $\switch(i,b,a)$ are in $S$.  Therefore if we
fix $i\leq t$, the elements in $S$ naturally form a matrix $M_i$ with
rows indexed by the $i$th
components and columns indexed by the
remaining $n-1$ components.
Furthermore, each generator $f_{i,a,b}$ of~$\CItS$
is a $2\times 2$ minor in $M_i$,
and $\CItS$ is generated
by all the $2\times 2$ minors of the matrices $M_1, \ldots, M_t$.
For each $i = 1, \ldots, t$,
let $s_i$ be the number of rows of $M_i$,
and let $s_{t+1}$ be the number of tuples that occur as the last
$n-t$ entries in elements in $S$.
Let $V_i$ be a vector space of dimension $s_i$ for $1\leq i\leq t+1$.
Then the $2\times 2$ minors of $M_i$
naturally cut out the rank 1 tensors in the flattenings
$V_i\tensor(\tensor_{j\neq i} V_j)$ of the tensor $V_1\tensor \cdots
\tensor V_t\tensor V_{t+1}$.
Hence by Discussion~\ref{discSegre},
the ideal $\CItS$ is a prime ideal.
\end{proof}

\begin{remark}\label{rmkCIbinommin}
Let $S$ be a $t$-switchable set and $\mathbf{x}_S= \prod_{a\in S}x_a$.
Define
$$\CIS = (f_{i,a,b} \mid i\in [t], a,b\ \text{connected in } S, d(a,b) = 2).$$
Discussion~\ref{discSegre} and the previous argument imply that
$\CItS = \CIS:\mathbf{x}_S$.  (Previously we knew this only for $S = N$.)
Thus the unique smallest binomial prime ideal that
contains $\CIS$ and that contains no monomials equals $\CItS$.
\end{remark}

\begin{thm}\label{thmCItSisprime}
If $S$ is a $t$-switchable set,
then the ideals $\CItS$ and $\CPS$ are prime.
\end{thm}

\begin{proof}
Partition $S = S_1 \cup \cdots \cup S_m$, into its equivalence classes
with respect to connectedness.
Therefore, the $S_i$ are pairwise disjoint.
Then $\CItS= \sum_{i=1}^m \CItSi$.
By Theorem~\ref{thmIStisprime},
each $\CItSi$ is a prime ideal.
Let $\overline R$ be the polynomial ring in the same variables as $R$
but over the algebraic closure of the underlying field.
Theorem~\ref{thmIStisprime} shows that each
$\CItSi \overline R$ is a prime ideal as well.
It is well known that in the polynomial ring
over an algebraically closed field,
if the generators of two prime ideals
are polynomials in disjoint sets of variables,
then the sum of the two prime ideals is also prime.
Thus $\CItS \overline R$ and $\CPS \overline R$ are prime ideals.
But the generators of these two prime ideals are in $R$,
so since $\overline R$ is a a faithfully flat extension of $R$,
$\CItS$ and $\CPS$ are
contractions of the prime ideals
$\CItS \overline R$ and $\CPS \overline R$ respectively,
and hence are prime themselves.
\end{proof}

Having established that $\CPS$ is prime, we set up those
$t$-switchable sets which correspond to minimal prime ideals for $\CI$.

\begin{defin}\label{defmaxtadm}
Let $S$ be $t$-switchable.
We say that $S$ is {\bf maximal $\bf t$-switchable}
if for all $t$-switchable subsets $T$ of $N$ properly containing $S$,
$\CPS$ and $\CPT$ are incomparable.
\end{defin}

\begin{remark}\label{rmkreach}
For brevity we state a few facts
but omit the straightforward proofs.
A maximal $t$-switchable set $S$ is not empty.
For all $a \in N \setminus S$,
$S \cup \{a\}$ is not $t$-switchable.
For every $i \in \{1, \ldots, n\}$ and every $u\in [r_i]$
there exists $b\in S$ such that $b_i = u$.
\end{remark}

\begin{prop}\label{propreach}
Let $S$ be a maximal $t$-switchable set
in which all elements of $S$ are pairwise connected.
Then $S = N$.
\end{prop}

\begin{proof}
Certainly $S \subseteq N$, and both $S$ and~$N$ are $t$-switchable.
We first prove that $\CPN \subseteq \CPS$.
Let $f_{i,a,b} \in \CPN$,
with $i \le t$.
If $a,b \in S$,
then using that $S$ is $t$-switchable and that all of its elements are
connected,
$f_{i,a,b} \in \CPS$.
So we may assume that either $a \not \in S$ or $b \not \in S$ and
similarly that either $\switch(i,a,b) \not \in S$ or $\switch(i,b,a) \not \in S$.
Hence $f_{i,a,b} \in \CVar \subseteq \CPS$.
Thus $\CPN \subseteq \CPS$,
and by maximality of $S$,
$S = N$.
\end{proof}

\begin{ex}\label{exnotmax}
The $1$-switchable set given in Example~\ref{exnotblock} is not
maximal as it is a single equivalence class, but is not all of $N$.
\end{ex}

\begin{lem}\label{lemReduction}
For $a_0, a_1, b \in N$ and $i \in [n]$,
$$
x_{a_1}f_{i,a_0,b} - x_{b}f_{i,a_0,a_1}
=
x_{\switch(i,a_0,a_1)}f_{i,\switch(i,a_1,a_0),b}
- x_{\switch(i,b,a_0)} f_{i,a_1,\switch(i,a_0,b)}.
$$
In particular,
if $a_0$ and $a_1$ differ only in the $l$th component and $l \not = i$,
then
$$
x_{a_1}f_{i,a_0,b} - x_{a_0}f_{i,a_1,b}
= - x_{\switch(i,b,a_0)} f_{i,a_1,\switch(i,a_0,b)}
\in (f_{i,a,b} : a, b \in N, d(a,b) = 2).
$$
More generally,
let $a_0, a_1, \ldots, a_k, b \in N$,
and assume that for all $j = 1, \ldots, l$,
$a_{j-1}$ and $a_j$ differ exactly in component $l_j \not = i$.
Then
$$
x_{a_1} x_{a_2} \cdots x_{a_k} f_{i,a_0,b}
-
x_{a_0} x_{a_1} \cdots x_{a_{k-1}} f_{i,a_k,b}
\in \sum_{j=1}^k (f_{i,a,b} : a, b \in N, d(a,b) = 2, a_{l_j} \not = b_{l_j}).
$$
In particular,
if $d(a_k,b) = 2$ and $a_k$ differ in components $i$ and $l_0$ for some $l_0$,
then
$$
x_{a_1} x_{a_2} \cdots x_{a_k} f_{i,a_0,b}
\in \sum_{j=0}^k (f_{i,a,b} : a, b \in N, d(a,b) = 2, a_{l_j} \not = b_{l_j}).
$$
\end{lem}

\begin{proof}
The first statement is straightforward rewriting:
\begin{align*}
x_{a_1}f_{i,a_0,b} - &x_{b}f_{i,a_0,a_1} -
x_{\switch(i,a_0,a_1)}f_{i,\switch(i,a_1,a_0),b} \\
&=
- x_{a_1} x_{\switch(i,a_0,b)} x_{\switch(i,b,a_0)}
+ x_{\switch(i,a_0,a_1)} x_{\switch(i,a_1,a_0)} x_{b}\\
&\hskip4em
- x_{\switch(i,a_0,a_1)}x_{b} x_{\switch(i,a_1,a_0)}
+ x_{\switch(i,a_0,a_1)}x_{\switch(i,b,\switch(i,a_1,a_0))}
x_{\switch(i,\switch(i,a_1,a_0),b)}\\
&=  x_{\switch(i,a_0,a_1)}
x_{\switch(i,b,\switch(i,a_1,a_0))} x_{\switch(i,\switch(i,a_1,a_0),b)}
- x_{a_1} x_{\switch(i,a_0,b)} x_{\switch(i,b,a_0)}\\
&= x_{\switch(i,a_0,a_1)} x_{\switch(i,b,a_0)} x_{\switch(i,a_1,b)}
- x_{a_1} x_{\switch(i,a_0,b)} x_{\switch(i,b,a_0)}\\
&= (x_{\switch(i,a_0,a_1)} x_{\switch(i,a_1,b)}
- x_{a_1} x_{\switch(i,a_0,b)}) x_{\switch(i,b,a_0)} \\
&= x_{\switch(i,b,a_0)}(-f_{i,a_1,\switch(i,a_0,b)}).
\end{align*}
If $d(a_0,a_1) = 1$,
then $f_{i,a_0,a_1} = 0$ for all $i$,
and $a_1$ and $\switch(i,a_0,b)$ differ at most in the two components
$i$ and $l$,
so that $f_{i,a_1,\switch(i,a_0,b)} \in
(f_{i,a,b} : a, b \in N, d(a,b) = 2)$.

The rest is an easy induction on $k$,
with the previous part being the base case.
\end{proof}

\begin{thm}\label{thmCIminprimes}
If $P$ is a prime ideal minimal over $\CI$,
then $P = \CPS$ for some maximal $t$-switchable set $S$.
\end{thm}

\begin{proof}
We first note that $\CIt = \CPN$,
and this is a prime ideal by Theorem~\ref{thmCItSisprime}.

Now let $P$ be an arbitrary prime ideal minimal over $\CI$.
Let $S$ be the set of all $a \in N$
such that $x_a \not \in P$.
We know that $S$ is not empty,
for otherwise $P$ is the ideal generated by all the variables,
which properly contains the already established minimal prime ideal $\CIt$.

Let $a, b \in S$ have $d(a,b) = 2$ and $a_i\neq b_i$ for some $i\in[t]$.
Since $P$ contains $\CI$ and $i\in[t]$, $P$ contains $f_{i,a,b}  = x_a
x_b - x_{\switch(i,b,a)} x_{\switch(i,a,b)}$.
Since $a, b \in S$,
then $x_a x_b \not \in P$,
so that necessarily
$x_{\switch(i,b,a)} x_{\switch(i,a,b)} \not \in P$,
and hence
$\switch(i,b,a), \switch(i,a,b) \in S$.
This proves that $S$ is $t$-switchable,
and so by Proposition~\ref{propQscontainsCI},
$\CI \subseteq \CPS$,
and by Theorem \ref{thmCItSisprime},
$\CPS$ is a prime ideal.

We next prove that $\CPS \subseteq~P$.
By the construction of $S$, $\CVar \subseteq P$.
Let $f_{i,a,b} \in \CItS$,
with $i\in[t]$ and $a$ and $b$ connected in $S$.
By the definition of connectedness,
there exist elements $a_0 = a, a_1, \ldots, a_k \in S$
such that for all $j = 1, \ldots, k$,
$a_{j-1}$ and $a_j$ differ only in one component,
and $d(a_k,b) = 2$.
By Lemma~\ref{lemconnected},
we may choose such a path so that $a_k$ and $b$ differ in the $i$th component.
Then by Lemma~\ref{lemReduction},
$x_{a_1} \cdots x_{a_k} f_{i,a,b} \in \CItS \subseteq P$,
and since $x_{a_j} \not \in P$,
it follows that $f_{i,a,b} \in P$,
as desired.
Thus $\CI \subseteq \CPS \subseteq P$.  Since
$\CPS$ is a prime ideal, by minimality of $P$, $\CPS = P$.

Finally,
let $T$ be $t$-switchable and properly containining $S$.
Then $\CVarT \subsetneq \CVar$,
so that $\CPS \not \subseteq \CPT$.
By Proposition~\ref{propQscontainsCI},
$\CPT$ contains $\CI$,
by Theorem~\ref{thmIStisprime},
$\CPT$ is a prime ideal,
and this combined with the fact that $P = \CPS$ is minimal over $\CI$,
implies that $\CPT \not \subseteq \CPS$.
Therefore
$\CPS$ and $\CPT$ are incomparable.
Thus $S$ is a maximal $t$-switchable set.
\end{proof}

\begin{thm}\label{thmCmaxadmiprime}
Let $S$ be a maximal $t$-switchable set.
Then $\CPS$ is a minimal associated prime ideal of $\CI$.
\end{thm}

\begin{proof}
By Proposition~\ref{propQscontainsCI},
$\CPS$ is a prime ideal containing $\CI$.
Let $P$ be a prime ideal contained in $\CPS$ and minimal over $\CI$.
By Theorem~\ref{thmCIminprimes},
$P = \CPT$ for some maximal $t$-switchable set $T$.
Since $\CPT \subseteq \CPS$,
necessarily $\CVarT \subseteq \CVar$,
so that $S \subseteq T$.
But then comparability of $\CPS$ and $\CPT$ and maximality of $S$ force $S = T$.
\end{proof}

The last two results prove:

\begin{thm}\label{thmminprimescont}
The set of prime ideals minimal over~$\CI$
equals the set of ideals of the form~$\CPS$
as $S$ varies over maximal $t$-switchable sets.
\qed
\end{thm}

\begin{cor}\label{corITpart}
Let $S$ be a $t$-switchable set (maximal or not)
such that $\CPS$ is associated to~$\CI$.
Then $\CItS$ is contained in the $\CPS$-primary component of $\CI$.
\end{cor}

\begin{proof}
Let $f_{i,a,b} \in \CItS$.
So $a, b \in S$ are connected and $i \in [t]$.
By Remark~\ref{rmkCIbinommin},
$\left(\prod_{c\in S} x_c\right) f_{iab} \in \CI$.
By construction,
$\prod_{c\in S} x_c^m\notin \CPS$ for any $m$
and hence $f_{iab}$ is in the $\CPS$-primary component of $\CI$.
\end{proof}

We note that Corollary~\ref{corITpart}
holds for binomial ideals in characteristic $0$ in general
by Eisenbud--Sturmfels~\cite[Theorem 7.1']{ESbinom}.
Our results for the specific binomial ideals $\CI$
are independent of the characteristic.

In the rest of the section we present some atypical behavior for $t = 1$.
The next lemma helps connect maximal $1$-switchable sets to the
admissible bipartite graphs in~\cite{Fink}. 

\begin{lem}\label{lemmaFcolonx}
Let $n \ge 3$,
and let $o, o' \in N$ differ at most in the first component.
Then $\CIo : x_o = \CIo : x_{o'}$.
\end{lem}

\begin{proof}
By symmetry it suffices to prove that $\CIo : x_o \subseteq \CIo : x_{o'}$.
Since
$x_o (\CIo : x_o) = \CIo \cap (x_o) = (\CIo \cdot Y + (x_o (Y-1)))R[Y] \cap R$
for a variable $Y$,
it follows that $\CIo : x_o$
is generated by binomials of the form
$g = x_{a_1} \cdots x_{a_l} - x_{b_1} \cdots x_{b_l}$.
Note that $x_o g$ being in $\CIo$
is the same as saying that there exists a sequential rewriting of
$o, a_1, \ldots, a_r$
into $o, b_1, \ldots, b_r$ (after possibly reindexing $b_1, \ldots, b_r$),
such that at each step,
only two indices change by switching their first components.
This means that for each $i = 1, \ldots, r$,
$a_i$ and $b_i$ differ at most in the first component.
Thinking of $o$ as being in the $0$th place on the original list
$o, a_1, \ldots, a_r$,
and $a_i$ in the $i$th place, we record each step 
in the rewriting process as a transposition $(i,j)$ when we switch the
first components of the $i$th and $j$th indices on the list.

Let $w$ be the composition of all these transpositions.
If $\CIo : x_o \not \subseteq \CIo : x_{o'}$,
we may choose such a $g \in (\CIo : x_o) \setminus (\CIo : x_{o'})$
for which the composition $w$ takes the fewest number of transpositions
for all such choices of $o$, $o'$ and $g$.

If while undergoing the transpositions,
the first entry of $x_o$ in the initial $x_o x_{a_1} \cdots x_{a_r}$
moves back into the first entry of $x_o$
in the final $x_o x_{b_1} \cdots x_{b_r}$,
then similarly
the first entry of $x_{o'}$ in the initial $x_{o'} x_{a_1} \cdots x_{a_r}$
moves back into the first entry of $x_{o'}$
in the final $x_{o'} x_{b_1} \cdots x_{b_r}$,
so that $g \in \CIo : x_{o'}$,
as desired.
In particular, this is the case if $a_{i1} \not = o_1$ for all~$i$
or if $o_1$ is not moved by any transposition.
If some $i > 0$ is never used in any transposition,
then $a_i = b_i$,
and hence the reduction steps give $g/x_{a_i} \in \CIo : x_o$, which
has strictly smaller~degree. 
So we may assume that each $i = 0, \ldots, r$ is invoked by some transposition,
and thus by minimality there are at least two reduction steps.
If the last transposition in $w$ does not include $0$,
then all but the last transposition of $w$ takes 
$o, a_1, \ldots, a_r$
to $o, c_1, \ldots c_r$ for some $c_1, \ldots, c_r$.
Then by the first statement in this paragraph
$x_{b_1} \cdots x_{b_r} - x_{c_1} \cdots x_{c_r} \in \CIo : x_{o'}$,
and by induction on the number of steps in $w$,
$x_{a_1} \cdots x_{a_r} - x_{c_1} \cdots x_{c_r} \in \CIo : x_{o'}$.
It follows that
$x_{a_1} \cdots x_{a_r} - x_{b_1} \cdots x_{b_r} \in \CIo : x_{o'}$.
Thus we may assume that the last transposition in $w$ includes $0$.

With this set-up,
we now perform these same successive (transposition) steps of $w$
on $o',a_1, \ldots, a_r$.
If we reduce in this way to $o', b_1, \ldots, b_r$,
we have that $g \in\CIo : x_{o'}$
as desired.
Considering that $w$ takes $o, a_1, \ldots, a_r$
to $o, b_1, \ldots, b_r$,
if $o', a_1, \ldots, a_r$ does not become $o', b_1, \ldots, b_r$,
$w$ must take it to 
$o, b_1, \ldots, b_{k-1}, 
\switch(1,b_k,o'), b_{k+1}, \ldots b_r$
for some $k \in \{1, \ldots, r\}$ where $b_{k1} = o_1$ and $k$ is such
that in the transpostions involving $o$,  $o_1$ lands in the first entry of $b_k$.

Let $z$ be the composition of precisely those first consecutive transpositions
by which the first component $o_1$ of $o$
arrives in its final place in the $k$th position.
Since the last transposition in $w$ must involve $0$,
and since $b_k$ has the same first component as $o$,
by minimality the last transposition does not involve $k$.
Therefore $z \neq w$.
Note that $z^{-1} \circ w$ takes
$o, a_1, \ldots, a_r$ to $o, c_1, \ldots, c_r$,
whereby the first component $o_1$ of $o$
in $o, a_1, \ldots, a_r$
remains in the $0$th position in $o, c_1, \ldots, c_r$.
Therefore
$z^{-1} \circ w$ takes
$o', a_1, \ldots, a_r$ to $o', c_1, \ldots, c_r$,
so that
$x_{a_1} \cdots x_{a_r} - x_{c_1} \cdots x_{c_r} \in \CIo : x_{o'}$.
Furthermore,
$z$ takes $o, c_1, \ldots, c_r$ to $o, b_1, \cdots, b_r$
and since $z$ has strictly fewer number of steps than $w$,
$x_{c_1} \cdots x_{c_r} - x_{b_1} \cdots x_{b_r} \in \CIo : x_{o'}$ by~induction.
Therefore
$x_{a_1} \cdots x_{a_r} - x_{b_1} \cdots x_{b_r} \in \CIo : x_{o'}$,
which proves that $\CIo : x_o = \CIo : x_{o'}$.
\end{proof}

In an earlier version of our paper
we included a similar (but more complicated) proof
that $\CIo : x_o x_{o'} = \CIo : x_o$,
which implies that $\CIo$ is radical.
For brevity we omit our complicated proof
because recently Ay and Rauh~\cite{AR} proved, in a more straightforward way,
that $\CIo$ is radical for all $n$
via the square-free nature of the leading terms of a Gr\"obner basis.

\begin{lem}\label{lemmareach1}
Let $S$ be a maximal $1$-switchable set.
Then for all $a \in S$ and all $b \in N$,
$\switch(1,a,b) \in S$.
In other words,
any element in $N$ that differs from an element in $S$
in at most the first component
is also in $S$.
\end{lem}

\begin{proof}
Let $a \in S$ and $b \in N$.
Then $c = \switch(1,a,b)$ differs from $a$ in at most the first component.
By Lemma~\ref{lemmaFcolonx},
$\CIo : x_a = \CIo : x_c$.
Since $a \in S$,
it follows that $x_a \not \in \CPS$,
so that for all positive integers $m$,
$\CIo : x_c^m = \CIo : x_a^m$ is contained in $\CPS$.
Hence $\CIo :
x_c^{\infty} = \CIo : x_a^{\infty}\subseteq \CPS$ and therefore, $x_c
\not \in \CPS$. Hence $c \in S$.
\end{proof}

\begin{remark}\label{rmk4Fink}
(Connection with admissible graphs in Fink~\cite{Fink}.)
Let $S$ be a maximal $1$-switchable
  set.  Lemma~\ref{lemmareach1} shows that
the first component is unrestricted in each equivalence class
determined by connectedness. Let $S = S_1\cup\cdots\cup S_l$ be a
partition of $S$ into equivalence classes.
We prove in the two paragraphs below that when $n = 3$,
$S_i = [r_1]\times T_{i2}\times T_{i3}$  where $T_{ij}\cap T_{kj}
= \emptyset$ for all $i\neq k$ and both $j = 2, 3$.
Therefore, each equivalence class
corresponds to the complete bipartite graph $T_{i2} \times T_{i3}$,
which is exactly Fink's representation in~\cite{Fink}.


We first argue that each equivalence class $S_i$ has the form $[r_1]\times
T_{i2}\times T_{i3}$.
Let $a = (a_1,a_2,a_3)$, $b = (b_1,b_2,b_3) \in S_i$.
It suffices to prove that
$\switch(2,b,a)$,
$\switch(3,b,a)$, $\switch(2,a,b)$, and $\switch(3,a,b)\in S_i$.
By Lemma~\ref{lemmareach1}
it suffices to consider the case where $a_1 = b_1$,
so that $d(a,b) \le 2$.  
If $d(a,b) = 1$ then each switch is either $a$ or $b$,
and the conclusion follows.
Now assume that $d(a,b) = 2$.
Suppose that $\switch(2,b,a)\notin S$.
Let $T = S \cup ([r_1] \times \{a_2\} \times \{b_3\})$.
We prove that $T$ is $1$-switchable.
Let $e, e' \in T$ satisfy $d(e,e') = 2$ and $e_1 \not = e'_1$.
By symmetry it suffices to prove that $\switch(1,e,e') \in T$.
If $e \in S$,
then $\switch(1,e,e')\in S \subseteq T$ by Lemma~\ref{lemmareach1},
and if $e \in T \setminus S$,
then $\switch(1,e,e')\in T$ by the definition of $T$.
Thus $T$ is $1$-switchable.
Using Lemma~\ref{lemmareach1}
it is also easy to see that $\CPT \subsetneq \CPS$,
which contradicts the maximality of $S$.
This proves that $\switch(2,b,a) \in S$,
and since it is connected to $b$,
it is in $S_i$.
Analogous proofs show that $\switch(3,a,b)$, $\switch(2,a,b)$,
$\switch(3,b,a)$ are in $S_i$.   
This proves that each equivalence class can be written in a ``block form"
$S_i = [r_1]\times T_{i2}\times T_{i3}$.
(By Example~\ref{exnotblock},
arbitrary $t$-switchable sets need not have a block form.)

Now suppose that $T_{i2}\cap T_{j2}\neq \emptyset$ for some distinct $i, j$.
Let $a_2\in T_{i2}\cap T_{j2}$.
Then by Lemma~\ref{lemmareach1}
there exist $a = (a_1, a_2, a_3)\in S_i$
and $b = (a_1, a_2, b_3)\in S_j$.
However, since $d(a,b) = 1$, $a$ and $b$ are
connected, which is a contradiction since they are in distinct
equivalence classes.  Therefore $T_{i2}\cap T_{j2}= \emptyset$.
Similarly,
$T_{i3}\cap T_{j3}= \emptyset$.
\end{remark}

\section{Prime Components for $t = n$}\label{sectprimecomponent}

In this section
we prove that the minimal components of $\CIn$
are all prime ideals.
Ay and Rauh~\cite{AR} prove that $\CIo$ is radical,
so that all the components of $\CIo$ are minimal and prime ideals.
In Section~\ref{sectexamples} we show by example 
that $\CIn$ may have embedded, and thus non-prime, components.
We consider it an interesting question to determine if the
minimal components of $\CI$ are prime for all $t$
(for examples, see Section~\ref{sectexamples}).

For all $t$,
by Theorem~\ref{thmminprimescont},
every prime ideal minimal over $\CI$ is of the form $\CPS = \CVar + \CItS$
for some maximal $t$-switchable set $S$.
By Corollary~\ref{corITpart},
the binomial portion of the $\CPS$-primary component of $\CI$ is $\CItS$.
Thus to prove that the minimal components of~$\CIn$ are prime,
it suffices to prove that $\CVarn$ is contained in the $\CPSn$-primary component.

We first prove that the equivalence classes of $n$-switchable sets
can be given in a block form.
By Remark~\ref{rmk4Fink},
when $t = 1$ and $n = 3$,
we also have a block form,
but an arbitrary $t$-switchable set need not have it
(see Example~\ref{exnotblock}).

\begin{thm}\label{thmblockdec}
Let $S$ be an $n$-switchable subset of $N$.
Let $T_1, \ldots, T_l$ be the equivalence classes for the relation of connected.
Then the following hold:
\begin{enumerate}
\item
Each $T_i$ is of the form $T_i = T_{i1} \times \cdots \times T_{in}$
for some $T_{ij} \subseteq [r_j]$.
\item
If $a \in T_i$, $b \in N$, and $\switch(j,a,b) \in S$,
then $\switch(j,a,b) \in T_i$.
\item
If $i, j$ are distinct in $[l]$,
there exist distinct $p_1, \ldots, p_m \in [n]$ with $m \ge 3$
such that for all $k = 1, \ldots, m$,
$T_{ip_k} \cap T_{jp_k} = \emptyset$.
\end{enumerate}
\end{thm}

\begin{proof}
For (1) it suffices to prove that whenever $a$ and $b$
are connected in $S$,
then for all $K \subseteq [n]$,
$\switch(K,a,b)$ is in $S$ and connected to $a$ and $b$.
But this is precisely Lemma~\ref{lemconnected}.

By Definition~\ref{defconnected},
$T_1, \ldots, T_l$ form a partition of $S$.

If $a \in T_i$, $b \in S$ such that  $\switch(j,a,b) \in S$,
then $d(a, \switch(j,a,b)) = 1$ so $a$ and $\switch(j,a,b)$ are
connected and therefore $\switch(j,a,b)\in T_i$.

Suppose that condition (3) fails.
By possibly reindexing,
$T_{1i} \cap T_{2i} \not = \emptyset$ for $i \ge 3$.
Let $a \in  T_1$, $b \in T_2$
such that $a_i = b_i$ for $i \ge 3$.
Then $1 \le d(a,b) \le 2$.
But then by Lemma~\ref{lemdist2conn},
$a$ and $b$ are connected,
so they
are both in $T_1 \cap T_2$,
giving a contradiction.
This proves~(3).
\end{proof}

By Theorem~\ref{thmIStisprime} ,
when we think of elements of $N$
as $(t+1)$-tuples rather than $n$-tuples
by reindexing $[r_{t+1}] \times \cdots \times [r_n]$
by $[r_{t+1} \cdots r_n]$,
since $\CInmin = \CIn$,
we have that the $t$-switchable sets in general have a block form.
However,
when the last $n-t$ components are spelled out explicitly,
there is not necessarily a block form;
see Example~\ref{exnotblock}.

Set the {\bf distance} between equivalence classes $T_i$ and $T_j$
to be the number of
indices $k$ such that $T_{ik}\cap T_{jk} = \emptyset$.
We denote this distance as $d(T_i, T_j)$, just as for elements of
$N$.
Thus, part (3) of the theorem above proves that $d(T_i, T_j) \ge 3$.
We note that the decomposition
does not require all components of two classes to be disjoint or equal,
but for each pair $T_i, T_j$
at least three have to be disjoint.

\begin{prop}\label{propMinDist} If $S \subseteq N$ is a maximal
  $n$-switchable set with equivalence classes $T_1,\ldots, T_k$,
 then for any class $T_j$ there exists another class $T_{l_j}$
  such that $d(T_j, T_{l_j}) = 3$.
\end{prop}

\begin{proof}
Without loss of generality $j = 1$.
By Theorem~\ref{thmblockdec} (3),
for all $l\neq 1$,
$d(T_1,T_l)\geq 3$.
Suppose for contradiction that for all $l\neq 1$,
$d(T_1,T_l) \geq 4$.
By possibly reindexing,
$T_{11} \cap T_{21} =~\emptyset$.
Let $u\in T_{21}\setminus T_{11}$.
Set $T_1' =
(T_{11}\cup\{u\}) \times T_{12}\times \cdots\times T_{1n}$,
and $S' =  T_1' \cup T_2\cup \cdots \cup T_k$.
We argue that $S'$ is $n$-switchable.
Let $a,b\in S'$ such that $d(a,b) = 2$.
For all $l\neq 1$,
$d(T_1, T_l)\geq 4$,
so that $d(T_1', T_l)\geq 3$.
By Theorem~\ref{thmblockdec} (3),
$d(T_i,T_l)\geq 3$ for all distinct $i,l$.
Hence $a$ and $b$ must either both be in $T_1'$
or they must both be in $T_i$ for some $i \ge 1$.
By the structure of the equivalence classes from Theorem~\ref{thmblockdec},
and by the definition of $T_1'$,
all the appropriate switches are contained in $S'$.
Hence $S'$ is still switchable.
But 
$\CPSn$ properly contains~$\CPSnP$,
which contradicts the maximality of $S$.
\end{proof}

\begin{lem}\label{lemMinDist}
Let $S$ be a maximal $n$-switchable set and $a\in N \setminus S$.
Then there exist $b,c\in S$ such that $d(a,b) \leq 2$,
$d(b,c) = 3$,
$b$ and $c$ are not connected,
and there exists $i \in [n]$
such that $c_i = a_i$ and $a_i\neq b_i$.
\end{lem}

\begin{proof}
(In this paragraph we allow $t$ to be arbitrary in $[n]$
and that $S$ is a maximal $t$-switchable set.)
Suppose that $d(a,b)\geq 3$ for all $b\in S$.
Set $T = S\cup\{a\}$.
The new element $a$ in $T$ is not connected in $T$
with any other element of $T$,
so that $T$ is $t$-switchable,
and if $c,d\in S$ are connected in $T$,
they must be connected in $S$.
But then $\CPS$ properly contains $\CPT$,
which contradicts the maximality of~$S$.
Hence there must exist $b\in S$ such that $d(a,b)\leq 2$.

We denote the equivalence class of all elements in $S$
that are connected to $b$ by $T_b$.
By Proposition~\ref{propreach},
$T_b \not = S$. Let $i_1,i_2$ denote the indices where $a$ and $b$
differ. Without loss of generality we assume $a_{i_1}\notin
T_{bi_1}$.  
Let $\{T_1,\ldots,T_l\}$ be the set of all equivalence classes of $S$
for which $d(T_i,T_b) = 3$.
By Proposition~\ref{propMinDist} this set is
non-empty.  We consider two cases. First, suppose there exists
$j\in\{1, \ldots l\}$ such that $a_{i_1}\in T_{ji_1}$ and
$T_{ji_1}\cap T_{bi_1} = \emptyset$.  Then $c$ exists as an element of
$T_j$.  Second, suppose that if $a_{i_1}\in T_{ji_1}$ then
$T_{ji_1}\cap T_{bi_1}\neq \emptyset$.  In this case, every element
$c\in S$ such that $d(b,c) = 3$ and $c$ and $b$ are not connected has
the property that $a_{i_1} \neq c_{i_1}$.  Build $T_b'$ from $T_b$ by
replacing $T_{bi_1}$ by $T_{bi_1}' = T_{bi_1}\cup \{a_{i_1}\}$.  Build
$S'$ by replacing $T_b$ in $S$ by $T_b'$.  We argue that the set  $S'$
is still $n$-switchable by arguing that the distance between $T'_b$
and any other equivalence class of $S$ is still at least~$3$.  Let $T$
be any equivalence class of $S$ with $T\neq T_b$.  Then either
$d(T,T_b) \geq 4$ and therefore $d(T,T_b')\geq 3$ or $T\in
\{T_1,\ldots T_l\}$.  In the latter case, either $a_{i_1}\notin
T_{i_1}$, in which case we still have $d(T, T_b') = 3$ or $a_{i_1}\in
T_{i_1}$ and then $T_{i_1}\cap T_b'\neq \emptyset$, but we assumed in
this case that $T_{i_1}\cap T_b\neq \emptyset$ so we still have that
$d(T,T_b) = 3$.  Since $S\subset S'$,  $\CVarnP\subset \CVarn$.  Let
$f_{j,e,o}\in \CInSP$.  Either $e,o$ are connected in $S$ or $o, e\in
T_b'$, $j = i_1$ and $o\notin T_b$, so $o_{i_1} = a_{i_1}$.  In the first case
$f_{j,e,o}\in \CInS$.  In the second case $o,\switch(i_1,e,o)\notin
S$, so $x_o,\switch(i_1,e,o)\in 
\CVarn$ and therefore $f_{i_1,e,o}\in \CVarn$.  Thus, $\CPSn$ contains
$\CPSnP$, which contradicts the  maximality of~$S$. Hence there exists
some $c \in S$ of distance $3$ from $b$ 
such that for some $i$,
$c_i = a_i \not = b_i$.
\end{proof}

\begin{lem}\label{reductionLemma2}
Let $a,b,c\in N$.
\begin{enumerate}
\item
If $a_i = c_i$, then
$x_af_{i,b,c} = x_cf_{i,a,b}-x_{\switch(i,b,a)}f_{i,c,\switch(i,a,b)}$.
\item
If $d(a,b) = 2$ and $d(b,c) = 3$,
$i \le t$,
$a_i = c_i$,
$b_i\neq a_i$,
and $b_j\neq c_j$ for some $j\neq i$ and $j\in [t]$,
then $x_af_{i,b,c}f_{j,b,c}\in \CI$.
\end{enumerate}
\end{lem}

\begin{proof}
\begin{enumerate}
\item This is straightforward rewriting:
\begin{align*}
x_af_{i,b,c} = & x_ax_bx_c - x_cx_{\switch(i,b,a)}x_{\switch(i,a,b)} +
x_cx_{\switch(i,b,a)}x_{\switch(i,a,b)} - x_a
x_{\switch(i,b,c)}x_{\switch(i,c,b)} \cr
= & x_ax_bx_c - x_cx_{\switch(i,b,a)}x_{\switch(i,a,b)} +
x_cx_{\switch(i,b,a)}x_{\switch(i,a,b)} - x_{\switch(i,\switch(i,a,b),c)}
x_{\switch(i,b,a)}x_{\switch(i,c,\switch(i,a,b))}\cr
= &  x_cf_{i,a,b} - x_{\switch(i, b,a)}f_{i,c,\switch(i,a,b)}.
\end{align*}
\item
If $d(a,b) = 2$, then $f_{i,a,b}\in \CI$.
Since $d(b,c) = 3$,
then $x_{\switch(i,b,c)}f_{j,b,c}\in \CI$
by Lemma \ref{lemReduction}.
Since $\switch(i,b,a) = \switch(i,b,c)$,
the conclusion follows from (1).
\end{enumerate}
\end{proof}

\begin{thm}\label{thmminimalcompisprime}
The minimal components of $\CIn$ are prime ideals.
\end{thm}

\begin{proof}
Let $Q$ be a minimal component for $\CIn$.
By Theorem~\ref{thmminprimescont},
the corresponding associated prime $P$ is of the form $\CPSn$
for some maximal $n$-switchable set $S$.
By Corollary~\ref{corITpart},
$\CInS \subseteq Q$.

Let $\mathbf{f}_S$  be the product of those elements $f_{i,a,b}\in \CItn$
that are not in $\CPSn$,
and let $\mathbf{x}_S = \prod_{a\in S} x_a$,
and $\rho_S = \mathbf{f}_S\mathbf{x}_S$.
Since~$\CPSn$ is a prime ideal,
it follows that $\rho_S\notin \CPSn$.
We prove that $\CPSn = \CIn:\rho_S$, and hence $Q = \CPSn$.

Let $x_a\in \CPSn$.
Then $a\notin S$.
By Lemma~\ref{lemMinDist},
there exist $b,c\in S$ such that $d(a,b) \leq 2$,
$b$ and $c$ are not connected,
$d(b,c) = 3$,
and for some $i$, $c_i = a_i\neq b_i$.
By construction $x_b, x_c\notin \CPSn$.
Since $b$ and $c$ are not connected,
then for all $j$ such that $b_j \not = c_j$,
by the structure of equivalence classes in Theorem~\ref{thmblockdec},
$\switch(j,b,c), \switch(j,c,b) \not \in S$.
Thus $f_{j,b,c}\notin \CPSn$.
Thus for any such $j$ different from $i$,
$x_b x_c f_{i,b,c}f_{j,b,c}$ is not in $\CPSn$.
If $d(a,b) = 1$,
then by our construction $a = \switch(i,b,c)$
and thus $x_af_{j,b,c}\in \CIn$ by Lemma~\ref{lemReduction},
and $x_a \in \CIn : f_{j,b,c} \subseteq \CIn : \rho_S$.
If $d(a,b) = 2$, then $x_af_{i,b,c}f_{j,b,c}\in
\CIn$ for all $j\neq i$, by Lemma~\ref{reductionLemma2} (2),
and so again here $x_a\in \CIn: \rho_S$.
Thus $\CPSn\subseteq \CIn:\rho_S$.  Since~$\CPSn$ is a prime
ideal containing $\CIn$,  $\CPSn\subseteq \CIn:\rho_S\subseteq
\CPSn:\rho_S = \CPSn$.
\end{proof}

\begin{remark}\label{rmkstructurerhoS}
We note that $\rho_S$ is a product of more factors than absolutely necessary
for a given ideal.
For example, when $S = N$ and $n = 3$,
then $\CIn : \prod_{k=0}^{r_3} x_{0,0,k}
= \CInN = \CPnN$,
so the proper factor $\prod_{k=0}^{r_3} x_{0,0,k}$ of~$\rho_N$ suffices.
On the other hand,
the proof of Theorem~\ref{thmminimalcompisprime}
contains an algorithm for determining
for general $S$
the diagonal minor factors $p$ of $\rho_S$ necessary to argue
$\CIn:p = \CPSn$.
These $p$ are not necessarily unique.
\end{remark}

\section{Gr\"obner Bases}\label{sectGB}

In this section we give a Gr\"obner basis for $\CItS$
in many monomial orders including the reverse lexicographic and the lexicographic order
under appropriate orderings of the variables.
This generalizes the work of~\cite{CGG} and~\cite{H}.

Recall that Caniglia, Guccione and Guccione~\cite{CGG}
proved that the $r \times r$ minors of a generic $m \times n$ matrix
form a Gr\"obner basis 
with respect to any ``diagonal order",
where diagonal orders on a two-dimensional matrix
are monomial orders such that for any $r \times r$ submatrix,
the product of the variables on the main diagonal
is the leading term of the minor of that submatrix. We start with an
example which illustrates some of the subtleties that arise in
extending the notion of diagonal orders for matrices to hypermatrices.
One $2\times 2$ minor of the  
$2\times 2\times 2$ hypermatrix is $x_{1,1,2} x_{2,2,1}- x_{1,2,2} x_{2,1,1}$.
We might order monomials by comparing
the first two components in the indices and following the notion of diagonal orders
from~\cite{CGG}, so that $x_{1,1,2} x_{2,2,1} > x_{1,2,2} x_{2,1,1}$.
However, we get a different order from comparing the last two
components in the indices, as the inequality reverses.
Therefore, a generalization of a diagonal order to an $n$-dimensional hypermatrix
must come equipped with a further prioritization of the components.  

\begin{defin}
We enumerate $[r_{t+1}] \times \cdots \times [r_n]$
and treat these last $n-t$ components of elements of $N$ as one component,
so we may assume that $t = n-1$.
(This also covers the case $t = n$.)
When we write $\switch(K,a,b)$ in this sense and $t+1 \in K$,
we actually mean $\switch([t] \setminus K,a,b)$ in the usual sense.
Let $\{\delta_1, \ldots, \delta_{t+1}\} = [t+1]$.
A {\bf $\bf t$-diagonal order} on $R$
relative to the enumeration $\delta_1, \ldots, \delta_{t+1}$
is any monomial order $<$
with the following property:
for any $a, b \in N$,
if $i$ is the smallest index in $[t+1]$ such that $a_{\delta_i} \not = b_{\delta_i}$,
and if $j > i$ such that $a_{\delta_j} > b_{\delta_j}$,
then $x_a x_b > x_{\switch(i,a,b)} x_{\switch(i,b,a)}$
if and only if $a_{\delta_i} > b_{\delta_i}$.
%
\end{defin}


For example,
the lexicographic order in which the variables are ordered
in the lexicographic order of their indices
is a $t$-diagonal order where $\delta_i = i$,
and the component $t+1$ in this sense stands for the last $n-t$ entries.
A reverse lexicographic order is a $t$-diagonal order
where $\delta_i = t-i+2$,
and the variables are ordered $x_a > x_b$
if $\sum_{i < j} |a_i - a_j| < \sum_{i < j} |b_i - b_j|$
or if $\sum_{i < j} |a_i - a_j| = \sum_{i < j} |b_i - b_j|$
and $a > b$ in the reverse lexicographic order.
Obviously many other options are possible.
When $n = 2$,
then any $t$-diagonal order
is a diagonal order for $2 \times 2$ minors as given in \cite{CGG}.

\begin{lem}\label{lmGBred}
Let $1 \le t \le n$,
let $S$ be an equivalence class with respect to connectedness
within some $t$-switchable set.
Set
$$
G = \{f_{K,a,b} : K \subseteq [t], a, b \in S\}.
$$
Suppose that $a_1, \ldots, a_r, b_1, \ldots, b_r \in S$
have the property that for all $i = 1, \ldots, t$,
up to order, the $r$-list $a_{1i}, a_{2i}, \ldots, a_{ri}$
is the same as the $r$-list $b_{1i}, b_{2i}, \ldots, b_{ri}$,
and such that, up to order, the $r$-list
$(a_{1,t+1}, a_{1,t+2}, \ldots, a_{1,n}),
(a_{2,t+1}, a_{2,t+2}, \ldots, a_{2,n}),
\ldots,
(a_{r,t+1}, a_{r,t+2}, \ldots, a_{r,n})$
is the same as the $r$-list
$(b_{1,t+1}, b_{1,t+2}, \ldots, b_{1,n}),
(b_{2,t+1}, b_{2,t+2}, \ldots, b_{2,n}),
\ldots,
(b_{r,t+1}, b_{r,t+2}, \ldots, b_{r,n})$.
Then in any $t$-diagonal monomial order,
$p = x_{a_1} x_{a_2} \cdots x_{a_r} - x_{b_1} x_{b_2} \cdots x_{b_r}$
reduces with respect to $G$ to $0$.
\end{lem}

\begin{proof}
We may assume that $p$ is reduced with respect to $G$.
%
If $r = 1$ or $n = 1$,
necessarily $p = 0$.
Now suppose that $r, n > 1$.
If the $\delta_1$ entries appearing in $a_1, \ldots, a_r$ are all the same,
the same holds for the $\delta_1$ entries in $b_1, \ldots, b_r$,
and by induction on $n$ the binomial $p$ reduces to $0$ with respect to
$\{f_{K,c,d} \in G: c_1 = d_1\}$.
So we may assume that
$a_{1\delta_1} = a_{2\delta_1} = \cdots = a_{s\delta_1} > a_{s+1,\delta_1} \ge \cdots \ge a_{r\delta_1}$
for some positive $s < r$.
If for some $i \in \{2, 3, \ldots, t+1\}$,
$j \le s$ and $l > s$,
the $\delta_i$ entry of $a_j$ is strictly bigger than the $\delta_i$ entry of $a_l$,
then $x_{a_1} \cdots x_{a_r}$ is not reduced with respect to $G$
as it can be reduced with respect to $f_{\delta_i,a_j,a_l} \in G$
(recall that in the context of diagonal orders,
$f_{t+1,a,b}$ stands for $f_{[t],a,b}$ in the usual sense).
But $p$ is assumed reduced,
which gives a contradiction.
So necessarily for all $i = 2, \ldots, t+1$,
all the minimal possible $\delta_i$ entries of $a_1, \ldots, a_r$
appear with correct multiplicities
as $\delta_i$ entries of $a_1, \ldots, a_s$.
Analogously,
$b_{1\delta_1} = b_{2\delta_1} = \cdots = b_{s\delta_1} = \max\{b_{j\delta_1}: j = 1, \ldots, r\}
> b_{s+1,\delta_1}, \ldots, b_{r\delta_1}$
and for each $i = 2, \ldots, t+1$,
all the minimal possible $\delta_i$ entries of $b_1, \ldots, b_r$
appear with correct multiplicities
as $\delta_i$ entries of $b_1, \ldots, b_s$.
Thus $a_1, \ldots, a_s, b_1, \ldots, b_s$
satisfy the conditions of the lemma,
and by necessity $a_{s+1}, \ldots, a_r, b_{s+1}, \ldots, b_r$
satisfy the conditions of the lemma.
By induction on~$r$,
up to reindexing,
by the reduced assumption,
$a_1 = b_1$,
$\ldots$, $a_s = b_s$,
$\ldots$, $a_r = b_r$.
Hence $p = 0$.
\end{proof}

The proof of the theorem above shows that if $t \ge n-1$,
then $p$ reduces to $0$ with respect to
$G = \{f_{i,a,b} : i \in [n], a, b \in S\}$.

\begin{thm}\label{thmGB}
Let $1 \le t \le n$,
and let $S$ be a $t$-switchable set.
Then the set
$$
G = \{f_{K,a,b} : K \subseteq [t], a, b \in S \text{ are connected}\}
$$
is a (non-minimal) Gr\"obner basis for $\CItS$ in any $t$-diagonal
monomial order.
\end{thm}

\begin{proof}
Write $S = S_1 \cup \cdots \cup S_r$,
where the $S_i$ are the equivalence classes with respect to connectedness.
Then $G = G_1 \cup \cdots \cup G_r$,
where
$G_i = \{f_{K,a,b} : K \subseteq [t], a, b \in S_i\}$.

If $a, b \in S_i$ and $K = \{k_1, \ldots, k_l\}$,
then
$f_{K,a,b} = \sum_{i=1}^l
f_{k_i,
\switch(\{k_1, \ldots, k_{i-1}\},a,b),
\switch(\{k_1, \ldots, k_{i-1}\},b,a)}$,
so that $G_i \subseteq \CItSi \subseteq~(G_i)$.
It follows that $G \subseteq \CItS \subseteq~(G)$.

If $f \in G_i$, $g \in G_j$ and $i \not = j$,
then the S-polynomial of $f$ and $g$ trivially reduces to $0$
with respect to $G$
because the variables appearing in $f$ are disjoint from the variables
appearing in $g$.
Observe that elements of $G$ and S-polynomials of two elements from the same $G_i$ are either $0$
or are binomials of the form as in Lemma~\ref{lmGBred},
hence they reduce with respect to~$G$ to $0$,
proving that $G$ forms a Gr\"obner basis of $\CItS$.
\end{proof}

By the remark after the previous theorem,
if $t \ge n-1$,
then the smaller set $\{f_{i,a,b}: i \in [n], a, b \in S \text{ are connected}\}$
is a Gr\"obner basis of $\CItS$.

Ha~\cite{H} proved the theorem above
in the reverse lexicographic order
with the lexicographic order on the variables.
In a previous version of this paper,
we used the theorem above,
together with the structure theory of binomial ideals,
to give a proof that $\CItS$ are prime ideals.
In this version we verified primeness of $\CItS$ via flattenings of tensors
in Theorem~\ref{thmIStisprime}.

\section{Examples}\label{sectexamples}

In this section we give examples
showing that the primary decomposition structure of the
conditional independence ideals $\CI$
can have embedded components.
All computations were
performed using the package Binomials~\cite{K}
in the program Macaulay2~\cite{GS}.

The first two examples show that $\CI$ is not radical in general.
Thus the ideals $\CI$ are different from the conditional independence ideals
in Herzog et al.~\cite{HHHKR}, Ohtani~\cite{O}, and Ay and Rauh~\cite{AR}.

\begin{ex}
Let $r_1 = r_2 = r_3 = r_4 = 2$ and $t = 3,4$.
This ideal has 26 components, of which 17 are minimal and 9 are embedded.
In particular, the maximal ideal, which is $P^{\langle t\rangle}_\emptyset$,
is associated, and contains every $\CPS$.
However, there are other prime ideals in between the minimal primes
and this maximal associated prime.
For example,
\begin{align*}
\CPS = (x_{(1,1,1,1)},
&
x_{(1,1,1,2)},
x_{(1,1,2,1)},
x_{(1,1,2,2)},
x_{(1,2,1,1)},
x_{(1,2,1,2)},
x_{(1,2,2,1)},
\\
 &
x_{(2,1,1,2)},
x_{(2,1,2,1)},
x_{(2,1,2,2)},
x_{(2,2,1,1)},
x_{(2,2,1,2)},
x_{(2,2,2,1)},
x_{(2,2,2,2)})\end{align*}
contains
\begin{align*}
\CPT = (
x_{(1,1,1,1)}, &
x_{(1,1,1,2)},
x_{(1,1,2,1)},
x_{(1,1,2,2)},
x_{(1,2,1,1)},
x_{(1,2,1,2)},\\
 &
x_{(2,1,2,1)},
x_{(2,1,2,2)},
x_{(2,2,1,1)},
x_{(2,2,1,2)},
x_{(2,2,2,1)},
x_{(2,2,2,2)}),
\end{align*}
where
$S = \{(1,2,2,2), (2,1,1,1)\}$ and
$T = \{(1,2,2,1), (1,2,2,2), (2,1,1,1), (2,1,1,2)\}$.
By~Theorem~\ref{thmminimalcompisprime},
the $\CPT$-minimal component is $\CPT$,
but the $\CPS$ component is much more complicated.
For example,
Macaulay2 gives it 101 generators.


Keeping $r_1 = r_2 = r_3 = r_4 = 2$
and changing $t$ to $2$,
the ideal $\CI$ has 31 minimal primes and 11 embedded primes
including the maximal ideal.
For $t =1 $ and the same $r_1, \ldots, r_4$,
the ideal $\CIo$ has 17 components,
and by work in~\cite{AR} these components are all primes and there are
no other components.
\end{ex}

The following example shows that $\CI$ is not even radical for $n = 3$.
Again, if $t=1$ the work in~\cite{AR} (and \cite{Fink} since $n = 3$) proves $\CIo$
is radical.  Therefore the  counterexample
uses $t = 2,3$.
Thomas Kahle brought the following example to our attention:

\begin{ex}
The simplest example is when $r_1 = r_2 = 2$ and $r_3 = 4$.
The ideal $\CIt$ with $t = 2,3$
has 29 minimal components and one embedded component associated to the
maximal ideal.  We note that the minimal components are all prime by
Theorem~\ref{thmminimalcompisprime}.
\end{ex}

We consider it an interesting open question as to
whether the minimal components are prime when $t\neq 1, n$. While the
$\CI$ are not lattice basis ideals (see Example~\ref{exItItt}),
one might first attack this
question by considering lattice basis ideals in general, or just those
with square-free terms.  However, the following example given to us by
Thomas Kahle shows that a general lattice basis ideal with square-free
terms does not have to have prime minimal components.

\begin{ex}
The ideal $(x_4x_8 - x_1x_9, x_4x_6- x_7x_9, x_2x_5 - x_3x_9, x_2x_3-x_5x_6)$
in variables $x_1, x_2, \ldots, x_9$ over a field
is equidimensional
and it has 6 components,
all of which are minimal and one of which is not prime.
\end{ex}

\section*{Acknowledgments}

We are grateful to Thomas Kahle for help with the examples.
We thank
Tony Geramita for pointing us to Tai Ha's paper~\cite{H},
and Seth Sullivant, Daniel Erman, Claudiu Raicu, and Mike Stillman
for helpful conversations about tensors.
We also thank the editor
and the anonymous referees for helping us streamline
and improve the presentation.


\end{document}